\documentclass[12pt]%
{amsart}
\usepackage{verbatim}
\usepackage[utf8x]{inputenc}

\usepackage{amssymb}
\usepackage{enumerate}

\newtheorem{theorem}{Theorem}

\newtheorem{lemma}[theorem]{Lemma}
\newtheorem{corollary}[theorem]{Corollary}
\newtheorem{fact}[theorem]{Fact}
\newtheorem{problem}[theorem]{Problem}

\theoremstyle{definition}

\newif\ifdeveloping

\ifdeveloping
\usepackage[notref,notcite]{showkeys}
\fi

\newcommand{\mc}[1]{\mathcal{#1}}
\newcommand{\mbb}[1]{\mathbb{#1}}

\newcommand{\mf}[1]{\mathfrak{#1}}

\newcommand{\setm}{\setminus}
\newcommand{\empt}{\emptyset}
\newcommand{\subs}{\subset}
\newcommand{\oo}{{{\omega}_1}}
\newcommand{\dom}{\operatorname{dom}}

\def\<{\left\langle}
\def\>{\right\rangle}
\def\cf{\operatorname{cf}}

\def\br#1;#2;{\bigl[ {#1} \bigr]^ {#2} }

\newcommand{\force}{\Vdash}

 \newcommand{\cb}{\mbb C}

\newcommand{\supp}{\operatorname{supp}}

\newcommand{\stickT}{%
\setbox255=\hbox{\raise1ex\hbox{$\hspace{0.2pt}\,\bullet\,$}}
\mathord{\rlap{\hbox to\wd255{\hss\hbox{$|$}\hss}}
\box255}
}
\newcommand{\stickS}{%
\setbox255=\hbox{\raise0.6ex\hbox{$\scriptstyle\bullet$}}
\mathord{\rlap{\hbox to\wd255{\hss\hbox{$\scriptstyle|$}\hss}}
\box255}
}
\newcommand{\stick}{{\mathchoice{\stickT}{\stickT}{\stickS}{\stickS}}}

\author[L. Soukup]{Lajos Soukup}
\thanks
  {
   The author was supported by the
Hungarian National Foundation for Scientific Research grants no.
K 68262 and K 61600. }

\address
      { Alfr{\'e}d R{\'e}nyi Institute of Mathematics, 
        Hungarian Academy of Sciences}
\email{soukup@renyi.hu}
\urladdr{http://www.renyi.hu/$\sim$soukup}
\subjclass[2000]{54A25, 03E35, 54A35  }
\keywords{almost disjoint, refinement, chromatic number , cardinality spectrum, Cohen reals, }

\title{Dense families of countable sets below $\mf c$}

\date{\today}
\begin{document}
\begin{abstract}
We show that it is consistent that $2^\omega$ is as large as you wish, 
and for each uncountable cardinal $\kappa\le 2^\omega$,
there are a set $T\in \br \mbb R;\kappa;$ and a family $\mc A\subs \br T;\omega;$ 
with $|\mc A|=\kappa$ such that 
\begin{enumerate}[(a)]
 \item $|\overline{A}\cap T|=\kappa$ for each $A\in \mc A$,
\item for each $X\in \br T;\omega_1;$ there is $A\in \mc A$
with $A\subs X$,
\end{enumerate}
and so (i) there is an almost disjoint family 
$\mc B\subs \br \kappa;\omega;$ with  size and  chromatic number $\kappa$, 
(ii) there is a locally compact, locally countable $T_2$ space  with
cardinality spectrum $\{\omega,\kappa\}$. 
\end{abstract}
\maketitle

Answering a question of Erdős and Hajnal, it was proved by Elekes and Hoffman \cite{ElHo}
that for every cardinal $\kappa$ there is an almost disjoint family $\mc A\subs \br 2^\kappa;\omega;$
with $\chi(\mc A)>\kappa$.

Concerning almost disjoint subfamilies of  
$\br 2^\omega;\omega; $, Komjáth \cite{Ko1} obtained a  stronger  
result: {\em there is an 
almost disjoint family
$\mc A\subs \br 2^\omega;\omega;$ which refines $ \br 2^\omega;\omega_1;$}
(i.e. for each $X \in \br 2^\omega;\omega_1; $ there is $A \in  \mc A$ with $A \subs X$),
and so $\mc A=2^\omega$.
What happens below the continuum?
Under MA${}_{\aleph_1}$ or in the Cohen model, if  $\mc A\subs \br \omega_1;\omega;$ is an 
almost disjoint family of size $\omega_1 $ then there is an uncountable set $X\subs \omega_1 $ such that 
$A\setm X$ is infinite for each $A\in \mc A$. However, we can prove the following result:
\begin{theorem}\label{tm:almost}
It is consistent that $2^\omega$ is as large as you wish, 
and for each uncountable cardinal $\kappa\le 2^\omega$, there is an almost disjoint family 
$\mc B\subs \br \kappa;\omega;$ of size $\kappa$ 
which  refines $\br \kappa;\omega_1;$, 
and so the chromatic number of $\mc B$ is $\kappa$.
\end{theorem}

Theorem 1 will follow easily from Theorem \ref{tm:phi} below.
However, 
to explain the topological origin of that result, first we should recall some definitions
and some results of Juhász and Weiss.

The {\em cardinality spectrum} $S(X)$ of any topological  space $X$ is the set of
cardinalities of all infinite closed subspaces of $X$.

The {\em density } of a family $\mc X$  of sets is defined as the cofinality of 
$\<\mc X,\supseteq\>$, i.e.
\begin{equation}\notag
d(\mc X)=\min\{|\mc A|:\mc A\subs \mc X\land \forall X\in \mc X \ 
(\exists A\in \mc A)\ A\subseteq X\}. 
\end{equation}
Denote $\mc C_\kappa$ the standard poset which introduces 
$\kappa$ Cohen reals, i.e. $\mc C_\kappa=Fn(\kappa,\omega;\omega)$. 

Juh\'asz and Weiss \cite{JW} proved the following theorems:
\begin{enumerate}[(JW1)]
 \item {\em If $\kappa \le 2^\omega$ and
$d(\br \kappa;{\omega_1};)=\kappa$ 
then there is a locally compact, locally countable $T_2$ space  $X$ with 
$S(X) =\{ \omega,\,\kappa \}$.}
\item {\em If $\cf(\kappa) = \omega$ and $\lambda^{\omega_1} < \kappa$
for $\lambda < \kappa$ then $V^{\mc C_\kappa}\models$ ``$d(\br \kappa;{\omega_1};)=\kappa$''.}
\end{enumerate}
It is unknown whether it is consistent that there are two cardinals
$\kappa\ne \kappa'<2^\omega$ with $\omega=\cf(\kappa)=\cf(\kappa')$
such that 
$d(\br \kappa;{\omega_1};)=\kappa$
and   $d(\br \kappa';{\omega_1};)=\kappa'$.

However, in (JW1) the full power of the assumption    
$d(\br \kappa;{\omega_1};)=\kappa$ is not needed.   
Given a cardinal $\kappa$ let us  say that the {\em principle $\Phi(\kappa)$} holds iff
there are a set $T\in \br \mbb R;\kappa;$ and a family $\mc A\subs \br T;\omega;$ 
with $|\mc A|=\kappa$ such that 
\begin{enumerate}[(a)]
 \item $|\overline{A}\cap T|=\kappa$ for each $A\in \mc A$,
\item for each $X\in \br T;\omega_1;$ there is $A\in \mc A$
with $A\subs X$. 
\end{enumerate}
It is easy to see that $d(\br \kappa;{\omega_1};)=\kappa$ implies 
$\Phi(\kappa)$, and  Juh\'asz and Weiss actually  proved the following statement: 
\begin{enumerate}[(JW1)]\addtocounter{enumi}{2}
 \item If $\kappa \le 2^\omega$ and $\Phi(\kappa)$ holds then 
there is a locally compact, locally countable $T_2$ space  $X$ with 
$S(X) =\{ \omega,\,\kappa \}$
\end{enumerate}

\begin{theorem}\label{tm:phi}
It is consistent that $2^\omega$ is as large as you wish, and
$\Phi(\kappa)$ holds for each uncountable $\kappa\le 2^\omega$. 
\end{theorem}
So it is consistent that $2^\omega$ is as large as you wish, and 
for every $\kappa \le  2^\omega$ there is a locally countable and locally compact $T_2$
space $X$ with $S(X) = \{\omega, \kappa\}$.
Let us remark that Juh\'asz and Weiss proved that there is a simpler way 
to obtain that consistency: 
\begin{enumerate}[(JW1)]\addtocounter{enumi}{3}
 \item {\em Suppose $V\models$ ``GCH'' and $ \lambda>\omega $ is a cardinal in $V$ . 
Then, in $V^{\mc  C_\lambda}$ ,
for every $\kappa \le  2^\omega$ there is a locally countable and locally compact $T_2$
space $X$ with $S(X) = \{\omega, \kappa\}$.}
\end{enumerate}

To prove Theorem \ref{tm:phi}  we should recall some definition from
\cite{FSSo}.
 In that paper  a new kind of side-by-side 
product of posets was introduced.
 Let $X$ be any set and 
$\<P_i:i\in X\>$ be a 
family of posets. For $p\in\prod_{i\in X}P_i$ the {\em support} of $p$ is defined by 
$\supp(p)=\{i\in X:p(i)\not=1_{P_i}\}$. 
Let $\prod^*_{i\in X}P_i$ be 
the set
\begin{displaymath}
\{{p\in\prod_{i\in X}P_i}:|{\supp(p)}|\le {\omega}
\}
\end{displaymath}
with the partial ordering
\begin{displaymath}
\begin{array}{@{}l@{}l}
p\leq q\quad\Leftrightarrow\quad
	&p(i)\leq_i q(i)\mbox{ for all }i\in X\mbox{ and}\\
	&\{{i\in X}:p(i)<_i q(i)<_i 1_{P_i}\}
		\text{ is finite }.
\end{array}
\end{displaymath}

If $P_i=P$ for some poset $P$ for every $i\in X$, we shall write 
 $\prod^*_{X}P$.

For $p$, $q\in\prod^*_{i\in X}P_i$ the relation $p\leq q$ can be 
represented as a combination of the two other distinct relations which we 
shall call horizontal and vertical, and denote by $\leq_h$ and $\leq_v$ 
respectively: 
\begin{eqnarray}\notag
 p\leq_h q\quad\Leftrightarrow&\quad \supp(p)\supseteq\supp(q)\text{
   and }
p\restriction\supp(q)= q, \\\notag
p\leq_v q\quad\Leftrightarrow&\quad
p\le q\text{ and }
	\supp(p)=\supp(q).
\end{eqnarray}

\begin{fact}
$p\le q$ iff there is $r$ such that $p\leq_h r \leq_v q$  
iff there is $t$ such that $p\leq_v t \leq_h q$.
\end{fact}

Instead of theorem \ref{tm:phi} we prove the following result.

\begin{theorem}\label{tm:phi2}
Assume GCH and let ${\mu}>{\omega}$ be any regular cardinal.
Let $P=\prod^*_{\mu}  Fn({\omega},2)$. Then forcing with $P$ preserves
cardinals
and cofinalities, $(2^{\omega})^{V^P}={\mu}$, and
\begin{equation}\notag
V^P\models 
\text{$\Phi(\kappa)$ holds for each $\omega_1\le{\kappa}\le 2^{\omega}$}.
\end{equation}
\end{theorem}

\begin{proof}[Proof of Theorem \ref{tm:phi2}]
By \cite[Corollary 2.4(a)]{FSSo} 
forcing with $P$ preserves every cardinal.
Since $P$ satisfies ${\omega}_2$-c.c. by \cite[Corollary
  2.4(a)]{FSSo},
we have $( 2^{\omega})^{V^P}\le ((|P|^{\oo})^{\omega})^V={\mu}$.

For ${\alpha}<{\mu}$, let ${\dot c}_{\alpha}$ be the 
$P$-name of the generic function 
from $\omega$ to $2$ added by the ${\alpha}$-th copy of $Fn(\omega,2)$ in 
$P$. Since the functions $\{c_{\alpha}:{\alpha}<{\mu}\}$ are pairwise
distinct  we have $(2^{\omega})^{V^P}={\mu}$.

Let $\cb=\{c_{\alpha}:{\alpha}<{\mu}\}$, and 
for $X\subs {\mu}$ write $\cb_X=\{c_{\xi}:{\xi}\in X\}$.

\begin{lemma}\label{lm:dense}
If  $c\in Fn({\omega},2)\setm \{\empt\}$,  $q\in P$,
$A\in \br {\mu};{\omega};$, ${\beta}\in {\mu}$ such that   
$q({\alpha})=c$ for each ${\alpha}\in A\cup \{{\beta}\}$ then 
\begin{displaymath}
q\force {\dot c}_{\beta}\in \overline{\{{\dot c}_{\alpha}:{\alpha}\in A\}}. 
\end{displaymath}
\end{lemma}

\begin{proof}
Assume that $q'\le q$, $\varepsilon\in Fn(\omega,2)$ and 
$q'\force  \dot c_\beta\in [\varepsilon]$. Then 
$\varepsilon\subs q'(\beta)$.
Since $|A|=\omega$, there is $\alpha\in A$ such that $q'(\alpha)=q(\alpha)=c$.
Define the condition $q''$ as follows:
\begin{equation}\notag
q''(\gamma)=\left\{
\begin{array}{ll}
q'(\beta)&\text{$\gamma=\alpha$,}\\
q'(\gamma)&\text{otherwise}.
\end{array}
\right. 
\end{equation} 
Then $q''$ is  a condition, moreover $q''\le q'$ because
$q''(\alpha)=q'(\beta)\supset q(\beta)=c=q'(\alpha)$.
Since
$ q''\force \varepsilon\subset \dot c_\alpha$, i.e.
 $ q''\force  \dot c_\alpha\in [\varepsilon]$, the Lemma is proved.
\end{proof}

To construct $T\in \br \mbb R;\kappa;$ and $\mc A\subs \br T;\omega; $
witnessing  $\Phi(\kappa)$ we need the following
lemma: 

\begin{lemma}\label{lm:oo}
\begin{displaymath}
V^P\models\ \forall A
\in \br {\mu};\oo;\ (\exists
I\in \br A;{\omega};\cap V) \
(\forall X\in \br \mu;\oo;)\ 
\overline{\cb_I}\cap \cb_X\ne\empt. 
\end{displaymath}
\end{lemma}

\begin{proof}[Proof of the lemma]
Assume that {\em $p\force $ ``$\dot A=\{{\dot {\alpha}}_{\nu}:{\nu}<\oo\}$''}.

Let $(p_\nu)_{\nu<\omega_1}$, $(q_\nu)_{\nu<\omega_1}$ be 
sequences of elements of $P$ and  
$(\alpha_{\nu})_{\nu<\omega_1}$ be a sequence of 
ordinals $<\omega_1$ such that 
\begin{enumerate}[(a)]
\item $p_0\leq p$ and $(p_\nu)_{\nu<\omega_1}$ is a descending 
sequence with respect to $\leq_h$;
\item \label{b} $q_\nu\leq_v p_\nu$ and 
$q_\nu\Vdash{\dot {\alpha}}_{\nu}={\check\alpha_\nu}$ 
for all $\nu<\omega_1$;
\item ${\alpha}_{\nu}\in \supp(q_{\nu})$;
\item \label{d}  $p_\nu\restriction S_\nu=q_\nu\restriction S_\nu$ for every 
$\nu<\omega_1$ where
\[ S_\nu=
\supp(q_\nu)\setminus(\supp(p)\cup\bigcup_{{\zeta}<\nu}\supp(q_{\zeta})).
\]\noindent
For $\nu<\omega_1$ let 
$u_\nu=
\{{\beta\in\supp(q_\nu)}:{q_\nu(\beta)\not=p_\nu(\beta)}\}$,   
\end{enumerate}
and $n_{\nu}=\dom q_{\nu}({\alpha}_{\nu})$.

 For $\nu<\omega_1$ let 
$d_\nu=\bigcup_{\beta<\nu}\supp(q_\beta)$. 
Then $(d_\nu)_{\nu<\omega_1}$ is a continuously increasing sequence 
in $\br {\mu};{\omega};$. 
Let
$u_\nu=\{{\beta\in\supp(q_\nu)}:{q_\nu(\beta)\not=p_\nu(\beta)}\}$ 
for $\nu<\omega_1$. Then  $u_\nu$ is finite  by (\ref{b}) and    
$u_\nu\subset d_\nu$ by (\ref{d}). Hence, by Fodor's lemma, there exists an 
uncountable  $Y\subseteq\omega_1$ such that  
$u_\nu=u^*$ for all $\nu\in Y$, for some fixed 
$u^*\in \br {\mu};{<{\omega}};$.

 Since $Fn({\omega},2)$ is countable, 
there exists an uncountable 
$Y'\subseteq Y$ such that  ${q_\nu\restriction u^*}=q^*$,
and $q_{\nu}({\alpha}_{\nu})=c$ 
for each ${\nu\in Y'}$. 

Let $Z\subs Y'$ with order type ${\omega}$. Write 
$Z=\{{\zeta}_n:n<{\omega}\}$ and
put $q=\cup \{q_{{\zeta}_n}:n<{\omega}\}$.
Then $q\leq_h q_{{\zeta}_n}$. Thus
\begin{displaymath}
q\force \{{\check\alpha}_{{\zeta}_n}:n<{\omega}\}\subs \dot A.  
\end{displaymath}
Pick $\beta\in X\setm \supp(q)$, and 
define the condition $q'$ as follows:
$\supp(q')=\supp (q)\cup\{\beta\}$, and 
\begin{equation}\notag
q'(\gamma)=\left\{
\begin{array}{ll}
c&\text{$\gamma=\alpha$,}\\
q(\gamma)&\text{otherwise}.
\end{array}
\right. 
\end{equation} 
Then $q'\le q$ is  a condition, moreover
\begin{displaymath}
q\force \dot c_\beta\in \overline{\{c_{{\alpha}_{{\zeta}_n}}:n<{\omega}\}} 
\end{displaymath}
by lemma \ref{lm:dense}.
\end{proof}

\begin{lemma}\label{lm:a}
For each 
${\omega}<{\kappa}<{\mu}$ and 
$I\in \br {\mu};{\kappa};$ there is 
$\mc {A}_I\subs \br I;{\omega};$
such that 
\begin{enumerate}[(i)]
\item    $|\mc {A}_I|\le{\kappa}$,
\item $\mc {A}_I$ refines $\br I;\oo;$,
\item $|\overline{\cb_A}\cap \cb|={\mu}$
for each $A\in \mc {A}_I$.
\end{enumerate}
\end{lemma}

\begin{proof}
By induction on ${\kappa}$.
If $cf({\kappa})>{\omega}$ then 
let
\begin{displaymath}
 \mc {A}_I=\{A\in \br I;{\omega};\cap V:|\overline{\cb_A}\cap \mathbb C|={\mu}\}. 
\end{displaymath}
Then $|\mc {A}|\le({\kappa}^{\omega})^V={\kappa}$.
(iii) holds just by the construction of $\mc {A}$.
(ii) holds by Lemma \ref{lm:oo}.

If $cf({\kappa})={\omega}$ then let 
$I=\bigcup\{I_n:n<{\omega}\}$ with $|I_n|<{\kappa}$, 
and put $\mc {A}_I=\bigcup\{\mc {A}_{I_n}:n<{\omega}\}$.
\end{proof}

We are ready to proof theorem \ref{tm:phi2}.
Clearly $T=\mbb R$   and $\mc A=\{A\in \br \mbb R;\omega;:\text{$A$ is crowded}\}$    witness
$\Phi(2^\omega)$ as it was observed by Komjáth \cite{Ko1}.

So we can consider an uncountable
cardinal ${\kappa}<2^{\omega}$. 
Define  an increasing sequence 
$\<T_n:n<\omega\>\subs \br {\mu};{\kappa};$
as follows.
Let $T_0={\kappa}$. 
If $T_n$ is defined consider the family
$\mc A_{T_n}\subs \br T_n;\omega;$ defined in Lemma \ref{lm:a}.
For each $A\in \mc A_{T_n}$ choose a set 
$Y_A\subs {\cb_A}$ of size $\kappa$,
and let $T_{n+1}=T_n\cup\bigcup\{Y_A:A\in \mc A_{T_n}\}$.

Then $T=\cup\{T_n:n<\omega\}$ and $\mc A=\cup\{\mc A_{T_n}:n<\omega\}$ work.
\end{proof}

\begin{proof}[Proof of Theorem \ref{tm:almost}]
Let $T\subs \br \mbb R;\kappa;$ and $\mc A\subs \br T;\omega;$ witness 
$\Phi(\kappa)$.
Pick distinct points $\{x_A\in T\cap \overline{A} :A\in \mc A\}$.
Since $|\mc A|=|\overline{A}\cap T|=\kappa$ for each $A\in \mc A$, we can do that.  
Choose a sequence  $B_A\in \br A;\omega;$ which  converges to $x_A$. 
Since $x_A\ne x_{A'}$ for $A\ne A'$, we have $|B_A\cap B_{A'}|<\omega$.
So $\mc B=\{B_A:A\in \mc A\}$ satisfies the requirements.
\end{proof}

In \cite{FSSo} for an uncountable cardinal $\lambda$ we defined
\begin{equation}\notag
\stick_\lambda=\min
\{|\mc X|:\mc X\subseteq\br \kappa;\omega;
\text{ s.t.  }
\forall y\in \br \kappa;\omega_1;\:\exists x\in X\;x\subseteq y
\}
\end{equation}
Although it was not explicitly proved in \cite{FSSo} we knew the following result:\\
{\em Assume GCH and let ${\mu}>{\omega}$ be any regular cardinal.
Then forcing with $P=\prod^*_{\mu} Fn({\omega},2)$ preserves
cardinals
and cofinalities, $(2^{\omega})^{V^P}={\mu}$, and
\begin{equation}\notag
V^P\models 
\text{$\stick_\kappa=\kappa$  for each $\omega_1\le{\kappa}\le 2^{\omega}$}.
\end{equation}
}\\
Actually, the family  $\mc A=\{A\in \br \kappa;\omega;\cap V :\sup A<\kappa\}$
witnesses $\stick_\kappa=\kappa$. Indeed, $|\mc A|=\kappa$ since GCH holds, and 
the proof of Lemma \cite[3.7]{FSSo} contains the argument which yields that 
$\mc A$ refines $\br \kappa;\omega_1;$.
It is clear that $\Phi(\kappa)$ implies $\stick_\kappa=\kappa$, but the 
reserve implication is unknown.
\begin{problem}
Is it true that   $\stick_\kappa=\kappa$ implies $\Phi(\kappa)$
for $\omega_1\le \kappa< 2^\omega$?
\end{problem}

Finally we remark that the assumption $\stick_\kappa=\kappa$ is enough to prove
Theorem \ref{tm:almost}.
Indeed,
for $\omega\le\kappa<2^\omega$ the Fr\'echet ideal $\mc I=\br \kappa;<\omega;$ is 
nowhere $\kappa^+$-saturated, so by   
\cite[Theorem 2.1]{BaHaMa}
every family $\mc D\subs \br \kappa;\omega;$ of size $\kappa$
is refined by an almost disjoint family.
Thus
\begin{corollary}
If $\stick_\kappa=\kappa$ for some $\kappa<2^\omega$
then some almost disjoint family $\mc B\subs \br \kappa;\omega;$ refines
$\br \kappa;\omega_1;$. 
\end{corollary}

\end{document}